\documentclass[10pt]{amsart}
\usepackage{amsmath, amsthm}
\usepackage{amsfonts}
\usepackage{xcolor,graphicx}

\newtheorem{theorem}{Theorem}

\newtheorem{example}[theorem]{Example}

\newtheorem{lemma}[theorem]{Lemma}

\begin{document}

\title{A Variant on the Feline Josephus Problem}
\author{Erik Insko and Shaun Sullivan}
\begin{abstract}
 In the Feline Josephus problem, soldiers stand in a circle, each having $\ell$ `lives'. Going around
the circle, a life is taken from every $k$th soldier; soldiers with 0 lives remaining are removed
from the circle. Finding the last surviving soldier proves to be an interesting and difficult
problem, even in the case when $\ell=1$. In our variant of the Feline Josephus problem, we
instead remove a life from $k$ consecutive soldiers, and skip 1 soldier. In certain cases, we find closed
formulas for the surviving soldier and hint at a way of finding such solutions in other cases.
\end{abstract}

\maketitle

\section{  History of Josephus Problem Variants}
The Josephus problem is named after Flavius Josephus, a Jewish soldier and historian living in the 1st 
century. According to Josephus' account of the siege of Jotapata, he and his 39 comrade soldiers were trapped in a cave, the exit of which was blocked by Romans. They chose suicide over capture and decided to form a circle and kill themselves by having every third soldier killed by his neighbor. 
Josephus states that by luck or maybe by the hand of God (modern scholars point out that Josephus was a well-educated scholar and predicted the outcome), he and another man remained the last two standing and gave up to the Romans
\cite{Jos70}.

Thus the original Josephus problem is to identify the last remaining soldier in the following elimination game, or more generally the order in which soldiers are eliminated from the game. 
A finite set of $n$ labeled soldiers are arranged in a circle 
and a skip length $k$ is set. 
In each round of the game $k$ consecutive soldiers are skipped 
and the $(k+1)$st soldier is eliminated.  To make the game as concrete as possible we assume that at each second a new soldier is either skipped or eliminated; thus each elimination round takes $k+1$ seconds, and we can talk about the time at which a soldier leaves the game.

The Josephus Problem and its variants have received considerable attention 
in recent years \cite{AF48, BS12, HH97,Jak73, Llo83,MW10, MNW07,OW91,Q11,Rob60, RW12,U05,U07,WW13,Woo73}.
For instance, in 1990 Odlyzko and Wilf gave an ``explicit-looking'' formula for  the 
survivor in a generalized Josephus problem where $k=2$ soldiers are skipped at a time.  
 In 2007 Matsumoto, Nakamigawa, and  Watanabe studied the \textit{Switchback Josephus Problem,} 
where the soldiers stand in a line, and every $k$th soldier is eliminated passing from left to right and back again
\cite{MNW07}.
Ruskey and Williams generalized the Josephus problem to the \emph{Feline Josephus Problem} by 
 introducing a uniform number of lives $\ell$, 
 so that elements are not eliminated until they have been selected for the $\ell$th 
 time. They proved two main results: 
The position of the surviving soldier stabilizes with respect to 
 increasing the number of lives, and when the number of soldiers $n$ and a soldier's position $j$ are fixed, there is a skip value $k$ that allows soldier $j$ to 
 be the survivor  for every value of $\ell$  \cite{RW12}.
In 2012, the second author and Beatty considered the \textit{Texas Chainsaw Josephus Problem} where all soldiers have one life, $k$ consecutive soldiers are eliminated and one is skipped \cite{BS12}; their main result is restated as Theorem \ref{onelife} in Section 
\ref{section2} below, and we extend this result by determining the order in which soldiers are eliminated from the Texas Chainsaw Josephus Game.  
In Section \ref{section3}, we consider 
the \textit{Feline Texas Chainsaw Josephus Problem} where each soldier has $\ell$ lives, similar to Ruskey and Williams' generalization \cite{RW12}.

\section{Texas Chainsaw Josephus Problem} \label{section2}

This section gives two results on a Josephus problem variant where $k$ consecutive soldiers are killed at a time,
and the next soldier is skipped. 

We will label the $n$ soldiers $0$ through $n-1$. 
Let $T(n,k)$ denote the remaining soldier at the end of the Texas Chainsaw Josephus Game.
For example, for  skip value $k=2$ and $n=10$ soldiers, the elimination order is 1, 2, 4, 5, 7, 8, 0, 3, 9; thus the sixth soldier would survive and $T(10,2)=6$. 
The following theorem gives the surviving soldier in this game.
\begin{theorem}\label{onelife} \cite[Corollary 5.1]{BS12}
If $n=a(k+1)^b+km$, where $n\equiv a\mod{k}$, $1\leq a\leq k$, and $b$ is as large as possible, then

\[
T(n,k)=(k+1)m.
\]
\end{theorem}

\begin{proof}
We first note that if $n\leq k$, then $m=0$, and the soldier labeled $0$ survives.

Next we consider the case when $n=a(k+1)^b$. In each elimination round, 
we remove $k$ out of every $k+1$ soldiers. 
Thus after traversing the circle once, there are exactly
\[
\dfrac{n}{k+1}=a(k+1)^{b-1}
\] soldiers remaining. 
After traversing the circle $b$ times, there will be $a\leq k$ soldiers remaining, 
and after each round we skip the soldier labeled $0$. Therefore, in this case $T(a(k+1)^b,k)=0$.

Finally, we consider the case when $n\neq a(k+1)^b$. Since we write $n=a(k+1)^b+km$, we must have $m\geq 1$.
We prove the result in this case by induction on $n>k$. 
After the first round, soldier $0$ is skipped, and $1,2,\ldots,k$ are eliminated. 
Soldier $k+1$ is next to be skipped, so we cyclically shift the labels back by $k+1$, 
so soldier $k+1$ is now labeled $0$, soldier 0 now has label $n-k-1$, and every other soldier $x$ is labeled $x-k-1$. 
Now we have $n-k$ soldiers remaining, and since $n-k\equiv a\mod{k}$, 
we have $n-k=a(k+1)^b+k(m-1)\geq a(k+1)^b$. 
By induction, $T(n-k,k)=(k+1)(m-1)$. Cyclically shifting back, we get $T(n,k)=(k+1)m$, as desired. \qedhere
\end{proof}

The previous theorem gives a concrete algorithm for finding the position of the lone surviver in the Texas Chainsaw Josephus problem. \\
\textbf{Algorithm:}
\begin{itemize}
 \item Find the $a \in \{1,2, \ldots, k\}$ such that $a \equiv n \bmod k$. 
 \item Find the largest integer $b$ such that $n-a(k+1)^b  \geq 0$.  
 \item Let $m = \displaystyle \frac{n-a(k+1)^b}{k}$.
 \item Then the surviving soldier in the Texas Chainsaw Josephus Game is $T(n,k) = (k+1)m $. 
\end{itemize}

\begin{example}
For example, if $n=605$ and $k=7$, then we calculate that \[ a\equiv 605 \bmod 7  \Rightarrow  a =3 ,\]  and the largest value of $b$ such that $605-3(8)^b \geq 0$ is $b=2$.
We take  $m= \frac{605-3(8)^2}{7} = 59$.
Finally we calculate that  $T(n,k)=472$.
\end{example}

One may wonder the order in which the soldiers are eliminated from the game.
The following theorem identifies the exact time in which soldier $x$ is eliminated from the game as a function of 
$k$ and $n$.

\begin{theorem} \label{thm:elim}
If $x=(k+1)m+s$ for $1\leq s\leq k$, then 
soldier $x$ is eliminated after $x-m$ seconds.
If $x=(k+1)m$, then there exist unique nonnegative integers $a$ and $b$ such that 
$n-km=a(k+1)^b$ where $a\not\equiv0\bmod(k+1)$ and so dividing $a$ by $k+1$, $a=(k+1)q+r$ where $0<r<k+1$. 
In this case, soldier $x$ is eliminated after $n-q$ seconds.
\end{theorem} 
 
\begin{proof}
If $x$ is not of the form $x=(k+1)m$, 
then $x$ is eliminated in first round. 
In the example following this proof, soldiers 1,2, and 3 are eliminated after 1,2, and 3 seconds.
Soldiers 5, 6, and 7 are eliminated after
4, 5, and 6 seconds, and so on. 
In general, it is easy to verify the first statement. 
The main part of this proof is to show when $x=(k+1)m$ is eliminated. 
This will be shown by a similar strong induction proof as Theorem 1.\\

If $n=1,2,\ldots,k+1$, the only soldier to consider is $x=0$, in which case $m=0$. When $n=1,2,\ldots,k$, then $a=n$, $b=0$, and $q=0$. This agrees with the fact that $x=0$ is eliminated last. When $n=k+1$, then $a=1$, $b=1$, and still $q=0$, again agreeing with the fact that soldier 0 is the last to be eliminated. Now that these base cases are verified, we move to the induction step of the proof.

If $n>k+1$, we begin by skipping soldier 0 and eliminating $1,2,\ldots,k$. This leaves the soldiers in order 
$k+1,k+2,\ldots,n-1,0$. 
Relabel them to $0,1,\ldots,n-k-1$, so if $x$ is in the original order, 
then $x^\prime=x-k-1$ in the relabeling, except if $x=0$, then $x^\prime=n-k-1$. 
This will be the two cases for the proof.
\begin{enumerate}
	\item[Case 1:] If $x=0$, then $x^\prime=n-k-1=(k+1)m+s$ for $0\leq s\leq k$. 
	If $s=0$, then there exists nonnegative integers $a$ and $b$ such that 
	$(n-k)-km=a(k+1)^b$ where $a\not\equiv0\bmod(k+1)$, 
	and so $a=(k+1)q+r$ where $0<r<k+1$. 
	The $x^\prime$th soldier is eliminated after $n-k-q$ seconds, and 
	so the $x$th soldier is eliminated after $n-q$ seconds in the original labeling. Notice that the calculation of $a$, $b$, and $q$ are the same for both $n-k$ and $n$ since the above equations imply that  

 \[
 n=(k+1)(m+1)+s \quad and \quad n-k(m+1)=a(k+1)^b.
 \] 
If $s\neq 0$, then the $x^\prime$th soldier is eliminated after $x^\prime-m$ seconds, 
and so the $x$th soldier is eliminated after $x-(m+1)$ seconds in the original labeling. 

	\item[Case 2:] If $x\neq0$, then $x^\prime=x-k-1=(k+1)(m-1)$ (In this case, $s$ is always 0). 
	By induction, there exists nonnegative integers $a$ and $b$ such that 
	$(n-k)-k(m-1)=a(k+1)^b$ where $a\not\equiv0\bmod(k+1)$, 
	and so $a=(k+1)q+r$ where $0<r<k+1$. 
	The $x^\prime$th soldier is eliminated at $n-k-q$, 
	and so the $x$th soldier is eliminated at $n-q$. Again, the calculation of $a$, $b$, and $q$ are the same for both $n-k$ and $n$ for the same reason as in Case 1. \qedhere

\end{enumerate}
\end{proof}

If we think of the soldiers in the Texas Chainsaw Josephus problem as cards in a standard 52 card deck, then predicting the elimination order gives a collection of interesting card tricks.  For instance, the following example simulates shuffling a deck of playing cards by repeatedly moving a card to the bottom and laying 3 cards down on the table.

\begin{example} 
Let $n=52$ and $k=3$, then the elimination order starts with $1,2,3,5,6,7,9,10,11,\ldots$ 
continuing to eliminate all number not of the form $4k$. 
After passing through the deck once, the elimination order continues as 

\[
4,8,12,20,24,28,36,40,44,0,16,32,48.
\]

Thus a magician can predict the last four cards remaining by having an audience member show the rest of the audience the first card, 17th card, 33rd card, and 49th card in that order, or 
the magician can ask an audience member to pick a favorite number between 1 and 52, 
show them the card, and then identify that card by playing the number of elimination rounds predicted by Theorem \ref{thm:elim}.
\begin{enumerate}
	\item[28:] The number $28=4(7)$, so $m=7$. So, $n-km=52-3(7)=31=(4(7)+3)(4)^0$. 
	Thus, $q=7$, which means $28$ is $45$th, or $7$th from the last.
	\item[16:] The number $16=4(4)$, so $m=4$. 
	So, $n-km=52-3(4)=40=(4(2)+2)(4)^1$. Thus, $q=2$, which means $16$ is $50$th, 
	or $2$nd from the last.
	\item[48:] The number $48=4(12)$, so $m=12$. 
	So, $n-km=52-3(12)=16=(4(0)+1)(4)^2$. Thus, $q=0$, which means $48$ is $52$th, or last.
\end{enumerate}

\end{example}

\section{A variant of the feline Josephus problem} \label{section3}
We now consider a situation where each soldier has a fixed number of $\ell$ `lives'. 
The soldiers do not get eliminated until they 
have been hit $\ell$ times. 
For example, with $n=7$ soldiers, $k=3$, and $\ell=2$, 
the soldier labeled 4 would remain after the following elimination process.

\begin{itemize}
 \item Round 1: 0,4 have two lives, and 1,2,3,5,6 have one life.
\item Round 2: 0,1,4,5 have one life, and 2,3,6 are eliminated.
\item Round 3: All but 4 have been eliminated.

\end{itemize}

Let $T(n,k,\ell)$ be the label of the soldier who survives the game where $k$ soldiers are eliminated in a row, and each soldier has $\ell$ lives.
\begin{lemma} For any $k$ and $\ell$, if $n\geq k$, then

\[
T((k+1)n,k,\ell)=(k+1)\cdot T(n,k,\ell)
\]

\end{lemma}

\begin{proof}

Suppose we start with $(k+1)n$ soldiers and begin the elimination. 
In the first $\ell$ rounds, the soldiers $1,2,\ldots,k,k+2,k+3,\ldots,(k+1)n-2,(k+1)n-1$ 
are eliminated and the soldiers $0,(k+1),2(k+1),\ldots,(n-1)(k+1)$ 
remain and still have $\ell$ lives. 
We now have $n$ soldiers, so rename the survivors $0,1,2,\ldots,n$. 
Let $T(n,k,\ell)$ be the surviving soldier in the new labelling, 
then the corresponding surviving soldier in the original labelling is $(k+1)T(n,k, \ell)$.
\end{proof}

We can obtain a similar reduction on the number of lives when $n$ and $k+1$ are relatively prime, and $\ell>k$.

\begin{lemma}\label{reducelives}

If $\gcd{(n,k+1)}=1$ and $\ell>k$, and $\ell\equiv\ell^{\prime}\bmod{k}$, then $T(n,k,\ell)=T(n,k,\ell^{\prime})$.

\end{lemma}

\begin{proof}
It suffices to show that after $k+1$ rounds, each soldier has been skipped exactly once, so the soldiers each have $\ell-k$ lives, and the $(k+2)$nd round starts by skipping soldier 0. Suppose for the sake of contradiction that soldier $i$ is skipped twice in the first $k+1$ rounds, say in rounds $s$ and $t$ with $1\leq s<t\leq k+1$. Since no soldier is eliminated, we can instead consider $(k+1)n$ soldiers labeled 0 through $(k+1)n-1$, round $j$ is represented by soldiers $(j-1)n$ through $jn-1$, and soldier $i$ is represented by the $k+1$ soldiers $i,i+n,i+2n,\ldots$. The soldiers skipped are the multiples of $k+1$, thus $i+(s-1)n$ and $i+(t-1)n$ are both multiples of $k+1$. This implies that $(t-s)n$ is a multiple of $k+1$, but since $\gcd{(n,k+1)}=1$, this would imply that $t-s$ is a multiple of $k+1$. This contradicts our earlier assumption that $1\leq s<t\leq k+1$. \end{proof}

Lemma \ref{reducelives} implies in particular that when $\ell=k+1$ the survivor $T(n,k, \ell)$  is 
the survivor of $T(n,k,1)=T(n,k)$, which means the case where $\ell=k$ is the maximum for $\ell$ 
that we need to consider. This is the first step in the proof of the following result for the survivor when $\ell=k$.

\begin{theorem} \label{thm:3}
If $\gcd(n,k+1)=1$ and $n>k$, then 
\[
T(n,k,k)=r+\dfrac{(k+1)^2\left(q-(k+1)^b\right)-i}{k}
\]
where $n=q(k+1)+r$ with $1\leq r\leq k$, $b=\left\lfloor \log_{k+1}{q}\right\rfloor$, 
and $i\equiv q - 1\bmod{k}$ with $0\leq i\leq k-1$.

\end{theorem}

\begin{proof}
First, consider the case $\ell=k+1$.  
Since the $(k+2)$nd round would have skipped soldier 0, 
we can identify the soldiers who would have been skipped in the $(k+1)$st round, 
who would also be the soldiers eliminated in the $k$th round of the $\ell=k$ case. 
Soldiers labeled $x=n-t(k+1)$ for $1\leq t<\frac{n}{k+1}$ are eliminated in the $(k+1)$st round. 
Let $q$ be the largest integer less than $\frac{n}{k+1}$, and let $r=n-q(k+1)$. Then $r-1$ is the smallest soldier eliminated in the $k$th round. 
Moreover, each soldier less than $r$ is eliminated in the $(k+1)$st round. 
Therefore, from this point on, $r, r+1,\ldots, r+k+1, r+k+2,\ldots, r+2k,\ldots,r+q(k+1)-2$ 
are the $qk$ soldiers remaining with one life each, and we are to start by skipping $r$. 
So, we can relabel to $0, 1, 2, \ldots, qk-1$, and find the survivor as $T(qk,k)$ using Theorem \ref{onelife}.

\begin{center}
\begin{tabular}{cccccccc}
$r$ & $r+1$  & $\cdots$ & $r+k-1$ & $r+k+1$ & $r+k+2$ & $\cdots$ & $n-1$ \\ 
$\downarrow$ & $\downarrow$ &     & $\downarrow$ & $\downarrow$ & $\downarrow$ & & $\downarrow$ \\
$0$ & $1$ &  $\cdots$  &  $k-1$ & $k$ & $k+1$ & $\cdots$ & $qk-1$ \\  
\end{tabular}
\end{center}

The relabeling map above is $r+t(k+1)+i\rightarrow tk+i$ for $0\leq i\leq k-1$ and $1\leq t<q$. 
Thus, we obtain the survivor $T(n,k,k)$ by reversing the relabeling of $T(qk,k)$.

Since $qk\equiv k\bmod{k}$, $T(qk,k)=(k+1)\left(q-(k+1)^b\right)$, 
where $b$ is defined for $n=qk$ in Theorem \ref{onelife}. 
To reverse the relabeling of $T(qk,k)=tk+i\rightarrow r+t(k+1)+i$, 
we must find $t$ and $i$. Since $0\leq i\leq k-1$ and $q -1 \equiv i\bmod{k}$, 
we can find $i$ uniquely by dividing $q$ by $k$. For $t$, we just solve $T(qk,k)=tk+i$ for $t$ to obtain
\[
t=\dfrac{(k+1)\left(q-(k+1)^b\right)-i}{k}.
\] Finally, we get
\[
T(n,k,k)=r+\dfrac{(k+1)\left(q-(k+1)^b\right)-i}{k}(k+1)+i,
\] which simplifies to the equation in the theorem statement.
\end{proof}

\begin{example}
The first and last three rounds of the game when $n=13$ and $\ell=k =4$ are depicted in 
Figure \ref{fig:slaughter} to illustrate the proof of Theorem \ref{thm:3}.

 \begin{figure} [h]
  \begin{center}
 \scalebox{.65}{
 \begin{picture}(200,180)(0,0)

\put(100,10){\color{black}{\circle{20}}}
\put(97,7){\color{black}{0}}
\put(65,20){\color{black}{\circle{20}}}
\put(62,17){\color{black}{1}}
\put(30,40){\color{black}{\circle{20}}}
\put(27,37){\color{black}{2}}
\put(10,70){\color{black}{\circle{20}}}
\put(7,67){\color{black}{3}}
\put(10,110){\color{black}{\circle{20}}}
\put(7,107){\color{black}{4}}
\put(30,140){\color{black}{\circle{20}}}
\put(27,137){\color{black}{5}}
\put(47,122){\vector(-1,1){10}}
\put(49,112){Skip}
\put(70,160){\color{black}{\circle{20}}}
\put(67,157){\color{black}{6}}
\put(130,160){\color{black}{\circle{20}}}
\put(127,157){\color{black}{7}}
\put(170,140){\color{black}{\circle{20}}}
\put(167,137){\color{black}{8}}
\put(190,110){\color{black}{\circle{20}}}
\put(187,107){\color{black}{9}}
\put(190,70){\color{black}{\circle{20}}}
\put(185,67){\color{black}{10}}
\put(170,40){\color{black}{\circle{20}}}
\put(164,37){\color{black}{11}}
\put(135,20){\color{black}{\circle{20}}}
\put(129,17){\color{black}{12}}

\put(70,25){\color{black}{\circle*{3}}}
\put(35,45){\color{black}{\circle*{3}}}
\put(15,75){\color{black}{\circle*{3}}}
\put(15,115){\color{black}{\circle*{3}}}
 
\put(70,25){\color{black}{\circle*{3}}}
\put(35,45){\color{black}{\circle*{3}}}
\put(15,75){\color{black}{\circle*{3}}}
\put(15,115){\color{black}{\circle*{3}}}
 
\end{picture} \hspace{1cm}
 \begin{picture}(200,180)(0,0)

\put(100,10){\color{black}{\circle{20}}}
\put(97,7){\color{black}{0}}
\put(65,20){\color{black}{\circle{20}}}
\put(62,17){\color{black}{1}}
\put(30,40){\color{black}{\circle{20}}}
\put(27,37){\color{black}{2}}
\put(10,70){\color{black}{\circle{20}}}
\put(7,67){\color{black}{3}}
\put(10,110){\color{black}{\circle{20}}}
\put(7,107){\color{black}{4}}
\put(30,140){\color{black}{\circle{20}}}
\put(27,137){\color{black}{5}}
\put(70,160){\color{black}{\circle{20}}}
\put(67,157){\color{black}{6}}
\put(130,160){\color{black}{\circle{20}}}
\put(127,157){\color{black}{7}}
\put(170,140){\color{black}{\circle{20}}}
\put(167,137){\color{black}{8}}
\put(190,110){\color{black}{\circle{20}}}
\put(187,107){\color{black}{9}}
\put(190,70){\color{black}{\circle{20}}}
\put(185,67){\color{black}{10}}
\put(165,70){\vector(1,0){15}}
\put(139,67){Skip}
\put(170,40){\color{black}{\circle{20}}}
\put(164,37){\color{black}{11}}
\put(135,20){\color{black}{\circle{20}}}
\put(129,17){\color{black}{12}}

\put(70,25){\color{black}{\circle*{3}}}
\put(35,45){\color{black}{\circle*{3}}}
\put(15,75){\color{black}{\circle*{3}}}
\put(15,115){\color{black}{\circle*{3}}}
 
\put(70,25){\color{black}{\circle*{3}}}
\put(35,45){\color{black}{\circle*{3}}}
\put(15,75){\color{black}{\circle*{3}}}
\put(15,115){\color{black}{\circle*{3}}}
\put(75,165){\color{black}{\circle*{3}}}
\put(135,165){\color{black}{\circle*{3}}}
\put(175,145){\color{black}{\circle*{3}}}
\put(195,116){\color{black}{\circle*{3}}}

\put(75,165){\color{black}{\circle*{3}}}
\put(135,165){\color{black}{\circle*{3}}}
\put(175,145){\color{black}{\circle*{3}}}
\put(195,116){\color{black}{\circle*{3}}}
\end{picture}}

  \vspace{12pt} 
  
   \scalebox{.65}{
\begin{picture}(200,180)(0,0)

\put(100,10){\color{black}{\circle{20}}}
\put(97,7){\color{black}{0}}
\put(65,20){\color{black}{\circle{20}}}
\put(62,17){\color{black}{1}}
\put(30,40){\color{black}{\circle{20}}}
\put(27,37){\color{black}{2}}
\put(53,63){\vector(-1,-1){15}}
\put(57,62){Skip}
\put(10,70){\color{black}{\circle{20}}}
\put(7,67){\color{black}{3}}
\put(10,110){\color{black}{\circle{20}}}
\put(7,107){\color{black}{4}}
\put(30,140){\color{black}{\circle{20}}}
\put(27,137){\color{black}{5}}
\put(70,160){\color{black}{\circle{20}}}
\put(67,157){\color{black}{6}}
\put(130,160){\color{black}{\circle{20}}}
\put(127,157){\color{black}{7}}
\put(170,140){\color{black}{\circle{20}}}
\put(167,137){\color{black}{8}}
\put(190,110){\color{black}{\circle{20}}}
\put(187,107){\color{black}{9}}
\put(190,70){\color{black}{\circle{20}}}
\put(185,67){\color{black}{10}}
\put(170,40){\color{black}{\circle{20}}}
\put(164,37){\color{black}{11}}
\put(135,20){\color{black}{\circle{20}}}
\put(129,17){\color{black}{12}}

\put(70,25){\color{black}{\circle*{3}}}
\put(35,45){\color{black}{\circle*{3}}}
\put(15,75){\color{black}{\circle*{3}}}
\put(15,115){\color{black}{\circle*{3}}}
 
\put(70,25){\color{black}{\circle*{3}}}
\put(35,45){\color{black}{\circle*{3}}}
\put(15,75){\color{black}{\circle*{3}}}
\put(15,115){\color{black}{\circle*{3}}}
\put(75,165){\color{black}{\circle*{3}}}
\put(135,165){\color{black}{\circle*{3}}}
\put(175,145){\color{black}{\circle*{3}}}
\put(195,116){\color{black}{\circle*{3}}}

\put(75,165){\color{black}{\circle*{3}}}
\put(135,165){\color{black}{\circle*{3}}}
\put(175,145){\color{black}{\circle*{3}}}
\put(195,116){\color{black}{\circle*{3}}}
\put(175,46){\color{black}{\circle*{3}}}
\put(140,26){\color{black}{\circle*{3}}}
\put(105,15){\color{black}{\circle*{3}}}
\put(60,25){\color{black}{\circle*{3}}}

\put(175,46){\color{black}{\circle*{3}}}
\put(140,26){\color{black}{\circle*{3}}}
\put(105,15){\color{black}{\circle*{3}}}
\put(60,25){\color{black}{\circle*{3}}}

\end{picture}  \hspace{1cm}
\begin{picture}(200,180)(0,0)

 \put(-10,-10){\color{black}{\line(0,1){190}}}
  \put(-10,-10){\color{black}{\line(1,0){220}}}
    \put(-10,180){\color{black}{\line(1,0){220}}}
    \put(210,-10){\color{black}{\line(0,1){190}}} 
\put(100,10){\color{black}{\circle{20}}}
\put(97,7){\color{black}{0}}
\put(65,20){\color{black}{\circle{20}}}
\put(62,17){\color{black}{1}}
\put(30,40){\color{black}{\circle{20}}}
\put(27,37){\color{black}{2}}
\put(10,70){\color{black}{\circle{20}}}
\put(7,67){\color{black}{3}}
\put(10,110){\color{black}{\circle{20}}}
\put(7,107){\color{black}{4}}
\put(37,110){\vector(-1,0){15}}
\put(40,107){Skip}
\put(30,140){\color{black}{\circle{20}}}
\put(27,137){\color{black}{5}}
\put(70,160){\color{black}{\circle{20}}}
\put(67,157){\color{black}{6}}
\put(130,160){\color{black}{\circle{20}}}
\put(127,157){\color{black}{7}}
\put(170,140){\color{black}{\circle{20}}}
\put(167,137){\color{black}{8}}
\put(190,110){\color{black}{\circle{20}}}
\put(187,107){\color{black}{9}}
\put(190,70){\color{black}{\circle{20}}}
\put(185,67){\color{black}{10}}
\put(170,40){\color{black}{\circle{20}}}
\put(164,37){\color{black}{11}}
\put(135,20){\color{black}{\circle{20}}}
\put(129,17){\color{black}{12}}

\put(70,25){\color{black}{\circle*{3}}}
\put(35,45){\color{black}{\circle*{3}}}
\put(15,75){\color{black}{\circle*{3}}}
\put(15,115){\color{black}{\circle*{3}}}
 
\put(70,25){\color{black}{\circle*{3}}}
\put(35,45){\color{black}{\circle*{3}}}
\put(15,75){\color{black}{\circle*{3}}}
\put(15,115){\color{black}{\circle*{3}}}
\put(75,165){\color{black}{\circle*{3}}}
\put(135,165){\color{black}{\circle*{3}}}
\put(175,145){\color{black}{\circle*{3}}}
\put(195,116){\color{black}{\circle*{3}}}

\put(75,165){\color{black}{\circle*{3}}}
\put(135,165){\color{black}{\circle*{3}}}
\put(175,145){\color{black}{\circle*{3}}}
\put(195,116){\color{black}{\circle*{3}}}
\put(175,46){\color{black}{\circle*{3}}}
\put(140,26){\color{black}{\circle*{3}}}
\put(105,15){\color{black}{\circle*{3}}}
\put(60,25){\color{black}{\circle*{3}}}

\put(175,46){\color{black}{\circle*{3}}}
\put(140,26){\color{black}{\circle*{3}}}
\put(105,15){\color{black}{\circle*{3}}}
\put(60,25){\color{black}{\circle*{3}}}
 \put(5,75){\color{black}{\circle*{3}}}
 \put(5,115){\color{black}{\circle*{3}}}
\put(35,145){\color{black}{\circle*{3}}}
 \put(65,165){\color{black}{\circle*{3}}}

 \put(5,75){\color{black}{\circle*{3}}}
 \put(5,115){\color{black}{\circle*{3}}}
\put(35,145){\color{black}{\circle*{3}}}
 \put(65,165){\color{black}{\circle*{3}}}
\put(165,145){\color{black}{\circle*{3}}}
 \put(185,116){\color{black}{\circle*{3}}}
 \put(195,76){\color{black}{\circle*{3}}}
 \put(165,46){\color{black}{\circle*{3}}}

\put(165,145){\color{black}{\circle*{3}}}
 \put(185,116){\color{black}{\circle*{3}}}
 \put(195,76){\color{black}{\circle*{3}}}
 \put(165,46){\color{black}{\circle*{3}}}

 \put(95,15){\color{black}{\circle*{3}}}
\put(70,15){\color{black}{\circle*{3}}}
\put(25,45){\color{black}{\circle*{3}}}
 \put(15,65){\color{black}{\circle*{3}}}

 \put(95,15){\color{black}{\circle*{3}}}
\put(70,15){\color{black}{\circle*{3}}}
\put(25,45){\color{black}{\circle*{3}}}
 \put(15,65){\color{black}{\circle*{3}}}
 
\put(25,145){\color{black}{\circle*{3}}}
 \put(75,155){\color{black}{\circle*{3}}}
\put(125,165){\color{black}{\circle*{3}}}
 \put(175,135){\color{black}{\circle*{3}}}
 
\put(25,145){\color{black}{\circle*{3}}}
 \put(75,155){\color{black}{\circle*{3}}}
\put(125,165){\color{black}{\circle*{3}}}
 \put(175,135){\color{black}{\circle*{3}}}
\put(185,76){\color{black}{\circle*{3}}}
\put(175,34){\color{black}{\circle*{3}}}
 \put(130,26){\color{black}{\circle*{3}}}
\put(105,5){\color{black}{\circle*{3}}}

\put(185,76){\color{black}{\circle*{3}}}
\put(175,34){\color{black}{\circle*{3}}}
 \put(130,26){\color{black}{\circle*{3}}}
\put(105,5){\color{black}{\circle*{3}}}
\put(35,35){\color{black}{\circle*{3}}}
\put(15,105){\color{black}{\circle*{3}}}
\put(35,135){\color{black}{\circle*{3}}}
\put(5,65){\color{black}{\circle*{3}}}

\put(35,35){\color{black}{\circle*{3}}}
\put(10,70){\color{black!40}{\circle{20}}}
\put(7,67){\color{black!40}{3}}
\put(5,65){\color{black!40}{\circle*{3}}}
\put(5,75){\color{black!40}{\circle*{3}}}
\put(15,65){\color{black!40}{\circle*{3}}}
\put(15,75){\color{black!40}{\circle*{3}}}

\put(15,105){\color{black}{\circle*{3}}}
\put(35,135){\color{black}{\circle*{3}}}
 \put(135,155){\color{black}{\circle*{3}}}
 \put(165,135){\color{black}{\circle*{3}}}
 \put(195,104){\color{black}{\circle*{3}}}
 \put(195,64){\color{black}{\circle*{3}}}

 \put(135,155){\color{black}{\circle*{3}}}
 \put(165,135){\color{black}{\circle*{3}}}
 \put(195,104){\color{black}{\circle*{3}}}
 \put(195,64){\color{black}{\circle*{3}}}
 
\put(170,140){\color{black!40}{\circle{20}}}
\put(167,137){\color{black!40}{8}}
\put(165,135){\color{black!40}{\circle*{3}}}
\put(175,135){\color{black!40}{\circle*{3}}}
\put(165,145){\color{black!40}{\circle*{3}}}
\put(175,145){\color{black!40}{\circle*{3}}}

\put(140,14){\color{black}{\circle*{3}}}
 \put(95,5){\color{black}{\circle*{3}}}
 \put(60,15){\color{black}{\circle*{3}}}
\put(25,35){\color{black}{\circle*{3}}}

\put(140,14){\color{black}{\circle*{3}}}

\put(100,10){\color{black!40}{\circle{20}}}
\put(97,7){\color{black!40}{0}}
\put(95,5){\color{black!40}{\circle*{3}}}
\put(95,15){\color{black!40}{\circle*{3}}}
\put(105,5){\color{black!40}{\circle*{3}}}
\put(105,15){\color{black!40}{\circle*{3}}}

\put(65,20){\color{black!40}{\circle{20}}}
\put(62,17){\color{black!40}{1}}
\put(60,15){\color{black!40}{\circle*{3}}}
\put(60,25){\color{black!40}{\circle*{3}}}
\put(70,15){\color{black!40}{\circle*{3}}}
\put(70,25){\color{black!40}{\circle*{3}}}

\put(30,40){\color{black!40}{\circle{20}}}
\put(27,37){\color{black!40}{2}}
\put(25,35){\color{black!40}{\circle*{3}}}
\put(25,45){\color{black!40}{\circle*{3}}}
\put(35,35){\color{black!40}{\circle*{3}}}
\put(35,45){\color{black!40}{\circle*{3}}}

\end{picture}}
  \vspace{12pt} 
  
   \scalebox{.65}{
 \begin{picture}(200,180)(0,0)

\put(100,10){\color{black}{\circle{20}}}
\put(97,7){\color{black}{0}}
\put(65,20){\color{black}{\circle{20}}}
\put(62,17){\color{black}{1}}
\put(30,40){\color{black}{\circle{20}}}
\put(27,37){\color{black}{2}}
\put(10,70){\color{black}{\circle{20}}}
\put(7,67){\color{black}{3}}
\put(10,110){\color{black}{\circle{20}}}
\put(7,107){\color{black}{4}}
\put(30,140){\color{black}{\circle{20}}}
\put(27,137){\color{black}{5}}
\put(70,160){\color{black}{\circle{20}}}
\put(67,157){\color{black}{6}}
\put(130,160){\color{black}{\circle{20}}}
\put(127,157){\color{black}{7}}
\put(170,140){\color{black}{\circle{20}}}
\put(167,137){\color{black}{8}}
\put(190,110){\color{black}{\circle{20}}}
\put(187,107){\color{black}{9}}
\put(190,70){\color{black}{\circle{20}}}
\put(185,67){\color{black}{10}}
\put(165,70){\vector(1,0){15}}
\put(139,67){Skip}
\put(170,40){\color{black}{\circle{20}}}
\put(164,37){\color{black}{11}}
\put(135,20){\color{black}{\circle{20}}}
\put(129,17){\color{black}{12}}

\put(70,25){\color{black}{\circle*{3}}}
\put(35,45){\color{black}{\circle*{3}}}
\put(15,75){\color{black}{\circle*{3}}}
\put(15,115){\color{black}{\circle*{3}}}
 
\put(70,25){\color{black}{\circle*{3}}}
\put(35,45){\color{black}{\circle*{3}}}
\put(15,75){\color{black}{\circle*{3}}}
\put(15,115){\color{black}{\circle*{3}}}
\put(75,165){\color{black}{\circle*{3}}}
\put(135,165){\color{black}{\circle*{3}}}
\put(175,145){\color{black}{\circle*{3}}}
\put(195,116){\color{black}{\circle*{3}}}

\put(75,165){\color{black}{\circle*{3}}}
\put(135,165){\color{black}{\circle*{3}}}
\put(175,145){\color{black}{\circle*{3}}}
\put(195,116){\color{black}{\circle*{3}}}
\put(175,46){\color{black}{\circle*{3}}}
\put(140,26){\color{black}{\circle*{3}}}
\put(105,15){\color{black}{\circle*{3}}}
\put(60,25){\color{black}{\circle*{3}}}

\put(175,46){\color{black}{\circle*{3}}}
\put(140,26){\color{black}{\circle*{3}}}
\put(105,15){\color{black}{\circle*{3}}}
\put(60,25){\color{black}{\circle*{3}}}
 \put(5,75){\color{black}{\circle*{3}}}
 \put(5,115){\color{black}{\circle*{3}}}
\put(35,145){\color{black}{\circle*{3}}}
 \put(65,165){\color{black}{\circle*{3}}}

 \put(5,75){\color{black}{\circle*{3}}}
 \put(5,115){\color{black}{\circle*{3}}}
\put(35,145){\color{black}{\circle*{3}}}
 \put(65,165){\color{black}{\circle*{3}}}
\put(165,145){\color{black}{\circle*{3}}}
 \put(185,116){\color{black}{\circle*{3}}}
 \put(195,76){\color{black}{\circle*{3}}}
 \put(165,46){\color{black}{\circle*{3}}}

\put(165,145){\color{black}{\circle*{3}}}
 \put(185,116){\color{black}{\circle*{3}}}
 \put(195,76){\color{black}{\circle*{3}}}
 \put(165,46){\color{black}{\circle*{3}}}

 \put(95,15){\color{black}{\circle*{3}}}
\put(70,15){\color{black}{\circle*{3}}}
\put(25,45){\color{black}{\circle*{3}}}
 \put(15,65){\color{black}{\circle*{3}}}

 \put(95,15){\color{black}{\circle*{3}}}
\put(70,15){\color{black}{\circle*{3}}}
\put(25,45){\color{black}{\circle*{3}}}
 \put(15,65){\color{black}{\circle*{3}}}
 
\put(25,145){\color{black}{\circle*{3}}}
 \put(75,155){\color{black}{\circle*{3}}}
\put(125,165){\color{black}{\circle*{3}}}
 \put(175,135){\color{black}{\circle*{3}}}
 
\put(25,145){\color{black}{\circle*{3}}}
 \put(75,155){\color{black}{\circle*{3}}}
\put(125,165){\color{black}{\circle*{3}}}
 \put(175,135){\color{black}{\circle*{3}}}
\put(185,76){\color{black}{\circle*{3}}}
\put(175,34){\color{black}{\circle*{3}}}
 \put(130,26){\color{black}{\circle*{3}}}
\put(105,5){\color{black}{\circle*{3}}}

\put(185,76){\color{black}{\circle*{3}}}
\put(175,34){\color{black}{\circle*{3}}}
 \put(130,26){\color{black}{\circle*{3}}}
\put(105,5){\color{black}{\circle*{3}}}
\put(35,35){\color{black}{\circle*{3}}}
\put(15,105){\color{black}{\circle*{3}}}
\put(35,135){\color{black}{\circle*{3}}}
\put(5,65){\color{black}{\circle*{3}}}

\put(35,35){\color{black}{\circle*{3}}}
\put(10,70){\color{black!40}{\circle{20}}}
\put(7,67){\color{black!40}{3}}
\put(5,65){\color{black!40}{\circle*{3}}}
\put(5,75){\color{black!40}{\circle*{3}}}
\put(15,65){\color{black!40}{\circle*{3}}}
\put(15,75){\color{black!40}{\circle*{3}}}

\put(15,105){\color{black}{\circle*{3}}}
\put(35,135){\color{black}{\circle*{3}}}
 \put(135,155){\color{black}{\circle*{3}}}
 \put(165,135){\color{black}{\circle*{3}}}
 \put(195,104){\color{black}{\circle*{3}}}
 \put(195,64){\color{black}{\circle*{3}}}

 \put(135,155){\color{black}{\circle*{3}}}
 \put(165,135){\color{black}{\circle*{3}}}
 \put(195,104){\color{black}{\circle*{3}}}
 \put(195,64){\color{black}{\circle*{3}}}
 
\put(170,140){\color{black!40}{\circle{20}}}
\put(167,137){\color{black!40}{8}}
\put(165,135){\color{black!40}{\circle*{3}}}
\put(175,135){\color{black!40}{\circle*{3}}}
\put(165,145){\color{black!40}{\circle*{3}}}
\put(175,145){\color{black!40}{\circle*{3}}}

\put(140,14){\color{black}{\circle*{3}}}
 \put(95,5){\color{black}{\circle*{3}}}
 \put(60,15){\color{black}{\circle*{3}}}
\put(25,35){\color{black}{\circle*{3}}}

\put(140,14){\color{black}{\circle*{3}}}

\put(100,10){\color{black!40}{\circle{20}}}
\put(97,7){\color{black!40}{0}}
\put(95,5){\color{black!40}{\circle*{3}}}
\put(95,15){\color{black!40}{\circle*{3}}}
\put(105,5){\color{black!40}{\circle*{3}}}
\put(105,15){\color{black!40}{\circle*{3}}}

\put(65,20){\color{black!40}{\circle{20}}}
\put(62,17){\color{black!40}{1}}
\put(60,15){\color{black!40}{\circle*{3}}}
\put(60,25){\color{black!40}{\circle*{3}}}
\put(70,15){\color{black!40}{\circle*{3}}}
\put(70,25){\color{black!40}{\circle*{3}}}

\put(30,40){\color{black!40}{\circle{20}}}
\put(27,37){\color{black!40}{2}}
\put(25,35){\color{black!40}{\circle*{3}}}
\put(25,45){\color{black!40}{\circle*{3}}}
\put(35,35){\color{black!40}{\circle*{3}}}
\put(35,45){\color{black!40}{\circle*{3}}}

\put(25,135){\color{black}{\circle*{3}}}
\put(65,155){\color{black}{\circle*{3}}}
\put(125,155){\color{black}{\circle*{3}}}
\put(185,104){\color{black}{\circle*{3}}}

\put(30,140){\color{black!40}{\circle{20}}}
\put(27,137){\color{black!40}{5}}
\put(25,135){\color{black!40}{\circle*{3}}}
\put(25,145){\color{black!40}{\circle*{3}}}
\put(35,135){\color{black!40}{\circle*{3}}}
\put(35,145){\color{black!40}{\circle*{3}}}

\put(70,160){\color{black!40}{\circle{20}}}
\put(67,157){\color{black!40}{6}}
\put(65,155){\color{black!40}{\circle*{3}}}
\put(65,165){\color{black!40}{\circle*{3}}}
\put(75,155){\color{black!40}{\circle*{3}}}
\put(75,165){\color{black!40}{\circle*{3}}}

\put(130,160){\color{black!40}{\circle{20}}}
\put(127,157){\color{black!40}{7}}
\put(125,155){\color{black!40}{\circle*{3}}}
\put(125,165){\color{black!40}{\circle*{3}}}
\put(135,155){\color{black!40}{\circle*{3}}}
\put(135,165){\color{black!40}{\circle*{3}}}

\put(190,110){\color{black!40}{\circle{20}}}
\put(187,107){\color{black!40}{9}}
\put(185,104){\color{black!40}{\circle*{3}}}
\put(185,116){\color{black!40}{\circle*{3}}}
\put(195,104){\color{black!40}{\circle*{3}}}
\put(195,116){\color{black!40}{\circle*{3}}}

 \put(5,105){\color{black}{\circle*{3}}}
 \put(165,34){\color{black}{\circle*{3}}}
 \put(130,14){\color{black}{\circle*{3}}}

\end{picture}
\hspace{1cm}
 \begin{picture}(200,180)(0,0)

\put(100,10){\color{black}{\circle{20}}}
\put(97,7){\color{black}{0}}
\put(65,20){\color{black}{\circle{20}}}
\put(62,17){\color{black}{1}}
\put(30,40){\color{black}{\circle{20}}}
\put(27,37){\color{black}{2}}
\put(10,70){\color{black}{\circle{20}}}
\put(7,67){\color{black}{3}}
\put(10,110){\color{black}{\circle{20}}}
\put(7,107){\color{black}{4}}
\put(30,140){\color{black}{\circle{20}}}
\put(27,137){\color{black}{5}}
\put(70,160){\color{black}{\circle{20}}}
\put(67,157){\color{black}{6}}
\put(130,160){\color{black}{\circle{20}}}
\put(127,157){\color{black}{7}}
\put(170,140){\color{black}{\circle{20}}}
\put(167,137){\color{black}{8}}
\put(190,110){\color{black}{\circle{20}}}
\put(187,107){\color{black}{9}}
\put(190,70){\color{black}{\circle{20}}}
\put(185,67){\color{black}{10}}
\put(170,40){\color{black}{\circle{20}}}
\put(164,37){\color{black}{11}}
\put(135,20){\color{black}{\circle{20}}}
\put(129,17){\color{black}{12}}

\put(70,25){\color{black}{\circle*{3}}}
\put(35,45){\color{black}{\circle*{3}}}
\put(15,75){\color{black}{\circle*{3}}}
\put(15,115){\color{black}{\circle*{3}}}
 
\put(70,25){\color{black}{\circle*{3}}}
\put(35,45){\color{black}{\circle*{3}}}
\put(15,75){\color{black}{\circle*{3}}}
\put(15,115){\color{black}{\circle*{3}}}
\put(75,165){\color{black}{\circle*{3}}}
\put(135,165){\color{black}{\circle*{3}}}
\put(175,145){\color{black}{\circle*{3}}}
\put(195,116){\color{black}{\circle*{3}}}

\put(75,165){\color{black}{\circle*{3}}}
\put(135,165){\color{black}{\circle*{3}}}
\put(175,145){\color{black}{\circle*{3}}}
\put(195,116){\color{black}{\circle*{3}}}
\put(175,46){\color{black}{\circle*{3}}}
\put(140,26){\color{black}{\circle*{3}}}
\put(105,15){\color{black}{\circle*{3}}}
\put(60,25){\color{black}{\circle*{3}}}

\put(175,46){\color{black}{\circle*{3}}}
\put(140,26){\color{black}{\circle*{3}}}
\put(105,15){\color{black}{\circle*{3}}}
\put(60,25){\color{black}{\circle*{3}}}
 \put(5,75){\color{black}{\circle*{3}}}
 \put(5,115){\color{black}{\circle*{3}}}
\put(35,145){\color{black}{\circle*{3}}}
 \put(65,165){\color{black}{\circle*{3}}}

 \put(5,75){\color{black}{\circle*{3}}}
 \put(5,115){\color{black}{\circle*{3}}}
\put(35,145){\color{black}{\circle*{3}}}
 \put(65,165){\color{black}{\circle*{3}}}
\put(165,145){\color{black}{\circle*{3}}}
 \put(185,116){\color{black}{\circle*{3}}}
 \put(195,76){\color{black}{\circle*{3}}}
 \put(165,46){\color{black}{\circle*{3}}}

\put(165,145){\color{black}{\circle*{3}}}
 \put(185,116){\color{black}{\circle*{3}}}
 \put(195,76){\color{black}{\circle*{3}}}
 \put(165,46){\color{black}{\circle*{3}}}

 \put(95,15){\color{black}{\circle*{3}}}
\put(70,15){\color{black}{\circle*{3}}}
\put(25,45){\color{black}{\circle*{3}}}
 \put(15,65){\color{black}{\circle*{3}}}

 \put(95,15){\color{black}{\circle*{3}}}
\put(70,15){\color{black}{\circle*{3}}}
\put(25,45){\color{black}{\circle*{3}}}
 \put(15,65){\color{black}{\circle*{3}}}
 
\put(25,145){\color{black}{\circle*{3}}}
 \put(75,155){\color{black}{\circle*{3}}}
\put(125,165){\color{black}{\circle*{3}}}
 \put(175,135){\color{black}{\circle*{3}}}
 
\put(25,145){\color{black}{\circle*{3}}}
 \put(75,155){\color{black}{\circle*{3}}}
\put(125,165){\color{black}{\circle*{3}}}
 \put(175,135){\color{black}{\circle*{3}}}
\put(185,76){\color{black}{\circle*{3}}}
\put(175,34){\color{black}{\circle*{3}}}
 \put(130,26){\color{black}{\circle*{3}}}
\put(105,5){\color{black}{\circle*{3}}}

\put(185,76){\color{black}{\circle*{3}}}
\put(175,34){\color{black}{\circle*{3}}}
 \put(130,26){\color{black}{\circle*{3}}}
\put(105,5){\color{black}{\circle*{3}}}
\put(35,35){\color{black}{\circle*{3}}}
\put(15,105){\color{black}{\circle*{3}}}
\put(35,135){\color{black}{\circle*{3}}}
\put(5,65){\color{black}{\circle*{3}}}

\put(35,35){\color{black}{\circle*{3}}}
\put(10,70){\color{black!40}{\circle{20}}}
\put(7,67){\color{black!40}{3}}
\put(5,65){\color{black!40}{\circle*{3}}}
\put(5,75){\color{black!40}{\circle*{3}}}
\put(15,65){\color{black!40}{\circle*{3}}}
\put(15,75){\color{black!40}{\circle*{3}}}

\put(15,105){\color{black}{\circle*{3}}}
\put(35,135){\color{black}{\circle*{3}}}
 \put(135,155){\color{black}{\circle*{3}}}
 \put(165,135){\color{black}{\circle*{3}}}
 \put(195,104){\color{black}{\circle*{3}}}
 \put(195,64){\color{black}{\circle*{3}}}

 \put(135,155){\color{black}{\circle*{3}}}
 \put(165,135){\color{black}{\circle*{3}}}
 \put(195,104){\color{black}{\circle*{3}}}
 \put(195,64){\color{black}{\circle*{3}}}
 
\put(170,140){\color{black!40}{\circle{20}}}
\put(167,137){\color{black!40}{8}}
\put(165,135){\color{black!40}{\circle*{3}}}
\put(175,135){\color{black!40}{\circle*{3}}}
\put(165,145){\color{black!40}{\circle*{3}}}
\put(175,145){\color{black!40}{\circle*{3}}}

\put(140,14){\color{black}{\circle*{3}}}
 \put(95,5){\color{black}{\circle*{3}}}
 \put(60,15){\color{black}{\circle*{3}}}
\put(25,35){\color{black}{\circle*{3}}}

\put(140,14){\color{black}{\circle*{3}}}

\put(100,10){\color{black!40}{\circle{20}}}
\put(97,7){\color{black!40}{0}}
\put(95,5){\color{black!40}{\circle*{3}}}
\put(95,15){\color{black!40}{\circle*{3}}}
\put(105,5){\color{black!40}{\circle*{3}}}
\put(105,15){\color{black!40}{\circle*{3}}}

\put(65,20){\color{black!40}{\circle{20}}}
\put(62,17){\color{black!40}{1}}
\put(60,15){\color{black!40}{\circle*{3}}}
\put(60,25){\color{black!40}{\circle*{3}}}
\put(70,15){\color{black!40}{\circle*{3}}}
\put(70,25){\color{black!40}{\circle*{3}}}

\put(30,40){\color{black!40}{\circle{20}}}
\put(27,37){\color{black!40}{2}}
\put(25,35){\color{black!40}{\circle*{3}}}
\put(25,45){\color{black!40}{\circle*{3}}}
\put(35,35){\color{black!40}{\circle*{3}}}
\put(35,45){\color{black!40}{\circle*{3}}}

\put(25,135){\color{black}{\circle*{3}}}
\put(65,155){\color{black}{\circle*{3}}}
\put(125,155){\color{black}{\circle*{3}}}
\put(185,104){\color{black}{\circle*{3}}} 


\put(30,140){\color{black!40}{\circle{20}}}
\put(27,137){\color{black!40}{5}}
\put(25,135){\color{black!40}{\circle*{3}}}
\put(25,145){\color{black!40}{\circle*{3}}}
\put(35,135){\color{black!40}{\circle*{3}}}
\put(35,145){\color{black!40}{\circle*{3}}}

\put(70,160){\color{black!40}{\circle{20}}}
\put(67,157){\color{black!40}{6}}
\put(65,155){\color{black!40}{\circle*{3}}}
\put(65,165){\color{black!40}{\circle*{3}}}
\put(75,155){\color{black!40}{\circle*{3}}}
\put(75,165){\color{black!40}{\circle*{3}}}

\put(130,160){\color{black!40}{\circle{20}}}
\put(127,157){\color{black!40}{7}}
\put(125,155){\color{black!40}{\circle*{3}}}
\put(125,165){\color{black!40}{\circle*{3}}}
\put(135,155){\color{black!40}{\circle*{3}}}
\put(135,165){\color{black!40}{\circle*{3}}}

\put(190,110){\color{black!40}{\circle{20}}}
\put(187,107){\color{black!40}{9}}
\put(185,104){\color{black!40}{\circle*{3}}}
\put(185,116){\color{black!40}{\circle*{3}}}
\put(195,104){\color{black!40}{\circle*{3}}}
\put(195,116){\color{black!40}{\circle*{3}}}

 \put(5,105){\color{black}{\circle*{3}}}
 \put(165,34){\color{black}{\circle*{3}}}
 \put(130,14){\color{black}{\circle*{3}}}


\put(170,40){\color{black!40}{\circle{20}}}
\put(164,37){\color{black!40}{11}}
\put(165,34){\color{black!40}{\circle*{3}}}
\put(165,46){\color{black!40}{\circle*{3}}}
\put(175,34){\color{black!40}{\circle*{3}}}
\put(175,46){\color{black!40}{\circle*{3}}}
\put(135,20){\color{black!40}{\circle{20}}}
\put(129,17){\color{black!40}{12}}
\put(130,14){\color{black!40}{\circle*{3}}}
\put(130,26){\color{black!40}{\circle*{3}}}
\put(140,14){\color{black!40}{\circle*{3}}}
\put(140,26){\color{black!40}{\circle*{3}}}
\put(10,110){\color{black!40}{\circle{20}}}
\put(7,107){\color{black!40}{4}}
\put(5,105){\color{black!40}{\circle*{3}}}
\put(5,115){\color{black!40}{\circle*{3}}}
\put(15,105){\color{black!40}{\circle*{3}}}
\put(15,115){\color{black!40}{\circle*{3}}}

\put(190,70){\color{black}{\circle{20}}}
\put(185,67){\color{black}{10}}
\put(195,64){\color{black}{\circle*{3}}}
\put(185,76){\color{black}{\circle*{3}}}
\put(195,76){\color{black}{\circle*{3}}}

\end{picture}
}
\caption{ First and Last Three Rounds } \label{fig:slaughter}
\end{center}
\end{figure}
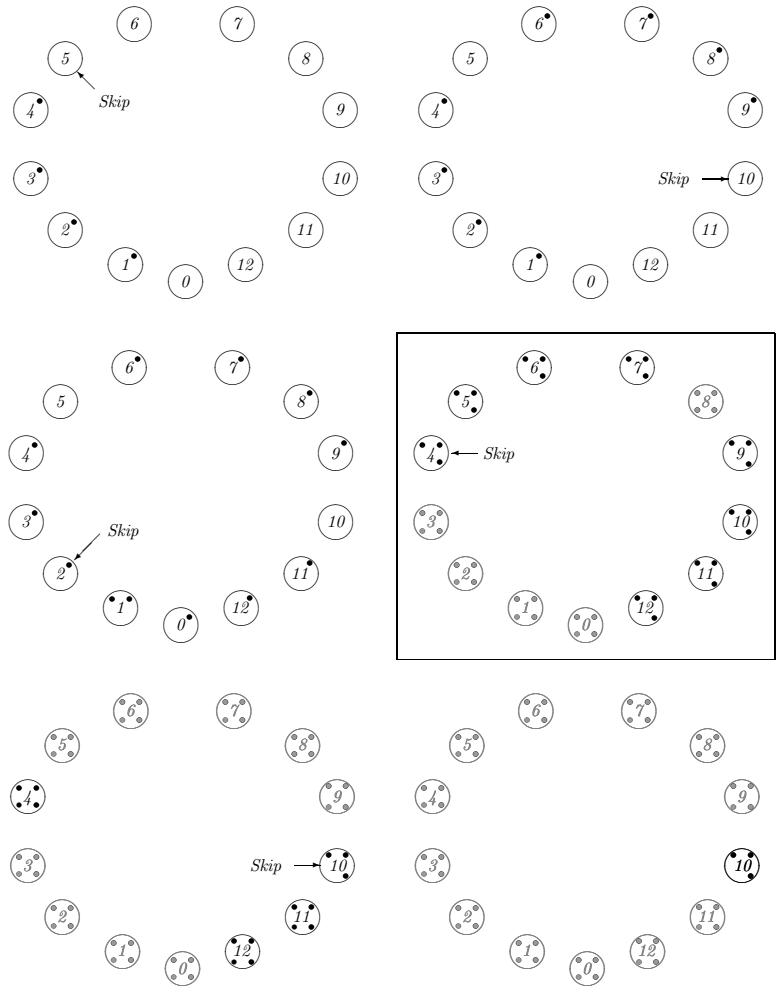

\newpage 
When each surviving soldier has only one life left to live (as highlighted in the box), 
there are eight soldiers remaining. We relabel the soldiers as in Figure \ref{fig:relabel} and use Theorem 
\ref{onelife} to predict the position of the surviving soldier to be position $5$.  
Finally, we reverse the relabeling to see that the soldier starting in position $10$ survives the game.
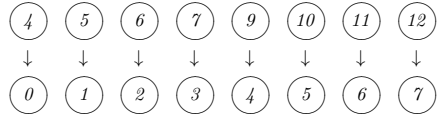
\begin{figure}[h]
\begin{center}
\scalebox{.7}{
 \begin{picture}(240,60)(0,0)
\put(10,50){\color{black}{\circle{20}}}
\put(7,47){\color{black}{4}}
\put(40,50){\color{black}{\circle{20}}}
\put(37,47){\color{black}{5}}
\put(70,50){\color{black}{\circle{20}}}
\put(67,47){\color{black}{6}}
\put(100,50){\color{black}{\circle{20}}}
\put(97,47){\color{black}{7}}

\put(130,50){\color{black}{\circle{20}}}
\put(127,47){\color{black}{9}}
\put(160,50){\color{black}{\circle{20}}}
\put(154,47){\color{black}{10}}
\put(190,50){\color{black}{\circle{20}}}
\put(184,47){\color{black}{11}}
	
\put(220,50){\color{black}{\circle{20}}}
\put(214,47){\color{black}{12}}
\put(7,27){$\downarrow$}
\put(37,27){$\downarrow$}
\put(67,27){$\downarrow$}
\put(97,27){$\downarrow$}
\put(127,27){$\downarrow$}
\put(157,27){$\downarrow$}
\put(187,27){$\downarrow$}
\put(217,27){$\downarrow$}
\put(10,10){\color{black}{\circle{20}}}
\put(7,7){\color{black}{0}}
\put(40,10){\color{black}{\circle{20}}}
\put(37,7){\color{black}{1}}
\put(70,10){\color{black}{\circle{20}}}
\put(67,7){\color{black}{2}}
\put(100,10){\color{black}{\circle{20}}}
\put(97,7){\color{black}{3}}

\put(130,10){\color{black}{\circle{20}}}
\put(127,7){\color{black}{4}}
\put(160,10){\color{black}{\circle{20}}}
\put(157,7){\color{black}{5}}
\put(190,10){\color{black}{\circle{20}}}
\put(187,7){\color{black}{6}}
\put(220,10){\color{black}{\circle{20}}}
\put(217,7){\color{black}{7}}
	
\end{picture}}
\end{center}
\caption{Relabeling Remaining Soldiers} \label{fig:relabel}
\end{figure}
\end{example}

\subsection{Algorithm for $\ell<k$:}

When $n$ and $k+1$ are relatively prime, at no point will the number of lives of any two soldiers differ by more than one. So when $\ell<k$, there will be a point where every soldier has exactly one life, similar to the $\ell=k$ case. The following algorithm gives this set of soldiers.

\begin{itemize}
 \item Let $S_1= \{ t : 0\leq t\leq \ell \cdot n -1 $ and $t\equiv 0 \bmod{(k+1)}\} $. 
 \item Let $S_2 = \{t \bmod{n}: t \in S_1 \}$. 
 \item Let $L = [\ell n -  (\ell n \mod (k+1)) ] \mod n$. 
 \item Finally, let $S$ be $S_2$ without the first $k-(n-L)$ members. 
 Every member of $S$ will have one life. 
\end{itemize}

In the $\ell=k$ case, we then relabeled the soldiers, used Theorem \ref{onelife}, and reversed the relabeling. When $\ell<k$, the relabeling map is not as simple, and so a simple formula is not expected here.

\begin{example}
When $n=13$, $\ell =3$, and $k=4$ then the algorithm gives the following set of soldiers with one life 
remaining as in Figure \ref{fig:slaughterK} on the next page.

\begin{center}
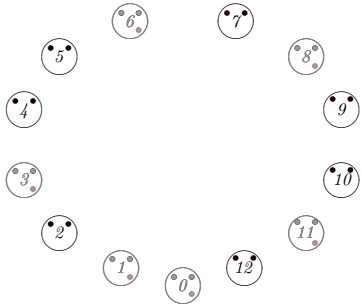
\begin{figure}[h] 
\scalebox{.6666}{
 \begin{picture}(200,170)(0,0)
\put(100,10){\color{black}{\circle{20}}}
\put(97,7){\color{black}{0}}
\put(65,20){\color{black}{\circle{20}}}
\put(62,17){\color{black}{1}}
\put(30,40){\color{black}{\circle{20}}}
\put(27,37){\color{black}{2}}
\put(10,70){\color{black}{\circle{20}}}
\put(7,67){\color{black}{3}}
\put(10,110){\color{black}{\circle{20}}}
\put(7,107){\color{black}{4}}
\put(30,140){\color{black}{\circle{20}}}
\put(27,137){\color{black}{5}}
\put(70,160){\color{black}{\circle{20}}}
\put(67,157){\color{black}{6}}
\put(130,160){\color{black}{\circle{20}}}
\put(127,157){\color{black}{7}}
\put(170,140){\color{black}{\circle{20}}}
\put(167,137){\color{black}{8}}
\put(190,110){\color{black}{\circle{20}}}
\put(187,107){\color{black}{9}}
\put(190,70){\color{black}{\circle{20}}}
\put(185,67){\color{black}{10}}
\put(170,40){\color{black}{\circle{20}}}
\put(164,37){\color{black}{11}}
\put(135,20){\color{black}{\circle{20}}}
\put(129,17){\color{black}{12}}

\put(70,25){\color{red}{\circle*{3}}}
\put(35,45){\color{red}{\circle*{3}}}
\put(15,75){\color{red}{\circle*{3}}}
\put(15,115){\color{red}{\circle*{3}}}
 
\put(70,25){\color{black}{\circle*{3}}}
\put(35,45){\color{black}{\circle*{3}}}
\put(15,75){\color{black}{\circle*{3}}}
\put(15,115){\color{black}{\circle*{3}}}
\put(75,165){\color{red}{\circle*{3}}}
\put(135,165){\color{red}{\circle*{3}}}
\put(175,145){\color{red}{\circle*{3}}}
\put(195,116){\color{red}{\circle*{3}}}

\put(75,165){\color{black}{\circle*{3}}}
\put(135,165){\color{black}{\circle*{3}}}
\put(175,145){\color{black}{\circle*{3}}}
\put(195,116){\color{black}{\circle*{3}}}
\put(175,46){\color{red}{\circle*{3}}}
\put(140,26){\color{red}{\circle*{3}}}
\put(105,15){\color{red}{\circle*{3}}}
\put(60,25){\color{red}{\circle*{3}}}

\put(175,46){\color{black}{\circle*{3}}}
\put(140,26){\color{black}{\circle*{3}}}
\put(105,15){\color{black}{\circle*{3}}}
\put(60,25){\color{black}{\circle*{3}}}
 \put(5,75){\color{red}{\circle*{3}}}
 \put(5,115){\color{red}{\circle*{3}}}
\put(35,145){\color{red}{\circle*{3}}}
 \put(65,165){\color{red}{\circle*{3}}}

 \put(5,75){\color{black}{\circle*{3}}}
 \put(5,115){\color{black}{\circle*{3}}}
\put(35,145){\color{black}{\circle*{3}}}
 \put(65,165){\color{black}{\circle*{3}}}
\put(165,145){\color{red}{\circle*{3}}}
 \put(185,116){\color{red}{\circle*{3}}}
 \put(195,76){\color{red}{\circle*{3}}}
 \put(165,46){\color{red}{\circle*{3}}}

\put(165,145){\color{black}{\circle*{3}}}
 \put(185,116){\color{black}{\circle*{3}}}
 \put(195,76){\color{black}{\circle*{3}}}
 \put(165,46){\color{black}{\circle*{3}}}

 \put(95,15){\color{red}{\circle*{3}}}
\put(70,15){\color{red}{\circle*{3}}}
\put(25,45){\color{red}{\circle*{3}}}
 \put(15,65){\color{red}{\circle*{3}}}

\put(95,15){\color{black}{\circle*{3}}}
\put(25,45){\color{black}{\circle*{3}}}

\put(65,20){\color{black!40}{\circle{20}}}
\put(62,17){\color{black!40}{1}}
\put(70,15){\color{black!40}{\circle*{3}}}
\put(60,25){\color{black!40}{\circle*{3}}}
\put(70,25){\color{black!40}{\circle*{3}}}

\put(10,70){\color{black!40}{\circle{20}}}
\put(7,67){\color{black!40}{3}}
\put(15,65){\color{black!40}{\circle*{3}}}
\put(15,75){\color{black!40}{\circle*{3}}}
\put(5,75){\color{black!40}{\circle*{3}}}
 
\put(25,145){\color{red}{\circle*{3}}}
 \put(75,155){\color{red}{\circle*{3}}}
\put(125,165){\color{red}{\circle*{3}}}
 \put(175,135){\color{red}{\circle*{3}}}

\put(25,145){\color{black}{\circle*{3}}}
\put(125,165){\color{black}{\circle*{3}}}

%
\put(70,160){\color{black!40}{\circle{20}}}
\put(67,157){\color{black!40}{6}}
\put(75,155){\color{black!40}{\circle*{3}}}
\put(75,165){\color{black!40}{\circle*{3}}}
\put(65,165){\color{black!40}{\circle*{3}}}

\put(170,140){\color{black!40}{\circle{20}}}
\put(167,137){\color{black!40}{8}}
\put(175,135){\color{black!40}{\circle*{3}}}
\put(165,145){\color{black!40}{\circle*{3}}}
\put(175,145){\color{black!40}{\circle*{3}}}

\put(185,76){\color{red}{\circle*{3}}}
\put(175,34){\color{red}{\circle*{3}}}
 \put(130,26){\color{red}{\circle*{3}}}
\put(105,5){\color{red}{\circle*{3}}}
%

\put(185,76){\color{black}{\circle*{3}}}

\put(130,26){\color{black}{\circle*{3}}}


\put(170,40){\color{black!40}{\circle{20}}}
\put(164,37){\color{black!40}{11}}
\put(175,34){\color{black!40}{\circle*{3}}}
\put(165,46){\color{black!40}{\circle*{3}}}
\put(175,46){\color{black!40}{\circle*{3}}}

\put(100,10){\color{black!40}{\circle{20}}}
\put(97,7){\color{black!40}{0}}
\put(105,5){\color{black!40}{\circle*{3}}}
\put(95,15){\color{black!40}{\circle*{3}}}
\put(105,15){\color{black!40}{\circle*{3}}}
 
\end{picture} }
\caption{Soldiers with one life remaining} \label{fig:slaughterK}
\end{figure}

\end{center}
\end{example}

\section{Open Problems}
We have only begun to scratch the surface of this variant, and so we list some interesting problems and questions that can be considered.

\begin{enumerate}
\item Given $k$, can we always find a non-trivial $n$ 
such that the same position survives for any 
$\ell$? This problem is similar to Ruskey and Williams's first main result \cite{RW12}. 
(For example, when $k=4$, let $n=13$ and the soldier labeled 10 survives for any $\ell$.) 
\item When $\gcd(n,k+1)=1$ and $k>n$, it appears that $T(n,k,k)=n$. 
\item We have ideas for reducing $\ell$ from $k$, but what about small $\ell$, i.e. $\ell=2$? 

\item This paper considers only when $n$ and $k+1$ are relatively prime in the Feline variant, but what about when $\gcd(n,k+1)\neq1$?
\end{enumerate}

\bibliography{JosephusBib} 

\bibliographystyle{plain}

\end{document}